\documentclass[12pt]{amsart}

\usepackage{fullpage}
\usepackage{amssymb}
\usepackage[utf8]{inputenc}
\usepackage{enumitem}

\theoremstyle{plain}

\newtheorem{theorem}{Theorem}

\newtheorem{lemma}[theorem]{Lemma}

\newtheorem{procedure}[theorem]{Procedure}
\theoremstyle{definition}

\newtheorem{definition}[theorem]{Definition}

\newtheorem{example}[theorem]{Example}
\theoremstyle{remark}

\newtheorem{remark}[theorem]{Remark}

\newcommand{\KK}{\mathbb K}
\newcommand{\PP}{\mathbb P}

\newcommand{\QQ}{\mathbb Q}

\newcommand\call{\mathcal L}

\newcommand\restr[1]{\big|_{#1}}

\DeclareMathOperator{\alphahat}{\widehat{\alpha}}

\DeclareMathOperator{\Ass}{Ass}

\begin{document}

\title{Lower bounds for Waldschmidt constants of generic lines in $\mathbb{P}^3$ and a Chudnovsky-type theorem}
\author{Marcin Dumnicki, Mohammad Zaman Fashami, Justyna Szpond,\\ Halszka Tutaj-Gasi\'nska}
\date{}

\begin{abstract}
  The Waldschmidt constant $\alphahat(I)$ of a radical ideal $I$ in the coordinate ring of $\PP^N$ measures (asymptotically)
  the degree of a hypersurface passing through the set defined by $I$ in $\PP^N$. Nagata's approach to the
  14th Hilbert Problem was based on computing such constant for the set of points in $\PP^2$. Since then, these constants drew much attention,
  but still there are no methods to compute them (except for trivial cases). Therefore the research focuses on looking
  for accurate bounds for $\alphahat(I)$.

  In the paper we deal with $\alphahat(s)$, the Waldschmidt constant for $s$ very general lines in $\PP^3$. We prove
  that $\alphahat(s) \geq \lfloor\sqrt{2s-1}\rfloor$ holds for all $s$, whereas the much stronger bound
  $\alphahat(s) \geq \lfloor\sqrt{2.5 s}\rfloor$ holds for all $s$ but $s=4$, $7$ and $10$.
  We also provide an algorithm which gives even better bounds for $\alphahat(s)$, very close to the known upper bounds,
  which are conjecturally equal to $\alphahat(s)$ for $s$ large enough.
\end{abstract}
\maketitle

\section{Introduction}

   In this note we study symbolic powers of ideals of finitely many very general lines in projective spaces. Our motivation comes from
   the general interest in asymptotic invariants of homogeneous ideals on the one hand and Chudnovsky-type questions relating
   the initial degree of an ideal to its Waldschmidt constant on the other hand. We discuss some methods leading to lower bounds on Waldschmidt
   constants of very general lines in $\PP^3$, which are reasonably close to conjecturally predicted exact values.

   Let $I\subset R=\KK[x_0,\ldots,x_N]$ be a homogeneous ideal. A celebrated result
   of Ein, Lazarsfeld and Smith \cite{ELS01} in characteristic zero
   and Hochster and Huneke \cite{HocHun02} in any characteristic asserts the containment
   \begin{equation}\label{eq:containment}
      I^{(m)}\subset I^r
   \end{equation}
   for all $m\geq rN$. Here $I^{(m)}$ denotes the $m^{th}$ symbolic power of $I$
   defined as
   $$I^{(m)}=R\cap \bigcap_{P\in\Ass(I)}I^mR_P,$$
   where the intersection is taken in the ring of fractions of $R$. In case the field $\KK$ is algebraically
   closed of characteristic $0$ and $I$ is a radical ideal from Zariski-Nagata Theorem \cite[Section 2]{DDGH} we have
   \begin{equation}\label{ZN}
   I^{(m)}=\left\{f: \frac{\partial^{|\alpha |}f}{\partial x^\alpha} \in I, |\alpha |\leq m-1\right\}.
   \end{equation}

   One of the fundamental invariants of a non-trivial homogeneous ideal $I$ is its \emph{initial degree}
   $$\alpha(I)=\min\left\{t:\; (I)_t\neq 0\right\},$$
   where $(I)_d$ denotes the degree $d$ part of $I$.
   The asymptotic version of the initial degree is the \emph{Waldschmidt constant}
   $$\alphahat(I)=\lim\limits_{m\to\infty}\frac{\alpha(I^{(m)})}{m}.$$
   It is well defined since the sequence of initial degrees of the symbolic powers
   of $I$ is sub-additive, see \cite[Lemma 2.3.1]{BocHar10JAG}.

   The containment result \eqref{eq:containment} implies the following lower bound for Waldschmidt constants of arbitrary homogeneous ideals in $N+1$ variables:
   $$\alphahat(I) \geq \frac{\alpha(I)}{N}.$$
   A better bound
   $$\alphahat(I) \geq \frac{\alpha(I)+1}{2}$$
   for ideals $I$ of points in $\PP^2$ is due to Chudnovsky \cite{Chu81}. Very recently,
   similar bounds have been proved for very general points in $\PP^N$.
   Dumnicki and Tutaj-Gasi\'nska \cite{DumTut17} and independently
   Fouli, Montero and Xie \cite{FMX18} proved that the lower bound
   $$\alphahat(I)\geq \frac{\alpha(I)+N-1}{N}$$
   holds for ideals of very
   general points in projective spaces of arbitrary dimension $N$.
   For ideals $I$ of very general points in $\PP^2$ and $\PP^3$ even better bounds for $\alphahat(I)$ are known, see \cite{DSS18} and \cite{Dum15}.

   The idea to pass from containment results for ideals of points in $\PP^N$ to higher dimensional flats
   has been exploited recently in \cite{GHV13}, see also \cite{SzeSzp17} for a survey on the containment problem.
   The article \cite{DHST14} studies asymptotic invariants of ideals supported on
   configurations of flats in the context of Nagata-type conjectures.
   The initial sequence for lines in $\PP^3$ has been studied by Janssen \cite{Jan15}. A natural line of continuing
   this approach is to study Waldschmidt constants of $s$ very general lines in $\PP^3$.
   From now on we denote these Waldschmidt constants by $\alphahat(s)$.

   In \cite[Theorem 2.5]{DHST14}, an upper
   bound for $\alphahat(s)$ 
   has been found.
   Namely, we have $\alphahat(s) \leq e_s$, where $e_s$ is the largest real root of the polynomial $\Lambda_s(t)=t^3-3st+2s$. For small values of $s$
   we have by \cite[Proposition B.2.1]{DHST14}
   \begin{center}
   \renewcommand*{\arraystretch}{1.2}
   \begin{tabular}{c|ccccccccccc}
   $s$ && 1 && 2 && 3 && 4 && 5\\
   \hline
   $\alphahat(s)$ && 1 && 2 && 2 && 8/3 && 10/3
   \end{tabular}.
   \end{center}
\medskip
   In this note we present three different approaches to bounding $\alphahat(s)$ from below.
   We provide a general bound in Theorem \ref{thm:approach1alg}. This allows us to derive a Chudnovsky-type statement for very general lines in $\PP^3$ in Theorem \ref{thm:chud}.
   Next, we show in Theorem \ref{thm:approach1} a general lower bound on $\alphahat(s)$ obtained by an elementary algorithm based
   on Theorem \ref{thm:approach1alg}. Considerably stronger results are obtained with a much more refined
   algorithm whose presentation fills Section \ref{sec:An algorithm} and culminates in Procedure \ref{proc:L}.

\section{Main results}
   Here we present our main results. The proofs fill the subsequent sections.
\begin{theorem}\label{thm:approach1alg}
Let $s$, $q$ and $k$ be positive integers satisfying
   \begin{equation}\label{eq:cond c k}
      (q-k)^2 \leq s-k^2.
   \end{equation}
Then $\alphahat(s) \geq q$.
\end{theorem}
   Theorem \ref{thm:approach1alg} provides an easy algorithm to bound $\alphahat(s)$. Indeed, for a fixed $s$ there are only finitely many
   pairs of integers $k$ and $q$ satisfying \eqref{eq:cond c k}. Taking the pair with the largest $q$ does the job.
   More effectively, we obtain the following bound expressed directly in $s$.
\begin{theorem}\label{thm:approach1}
   For all $s\geq 1$ there is
   \begin{equation}\label{eq:2s-1}
      \alphahat(s) \geq \lfloor \sqrt{2s-1} \rfloor.
   \end{equation}
\end{theorem}
   Working with more care, we get the following result.
\begin{theorem}\label{thm:approach2alg}
   Let $s$, $k$, $q$ be integers satisfying $qk \leq s$ and $(q-k)^2 \leq s-k$. Then $\alphahat(s) \geq q$.
\end{theorem}
   It is possible to determine the maximal $q$ satisfying conditions in Theorem \ref{thm:approach2alg}
   effectively in an algorithmic way.
   As a corollary, we obtain, with additional arguments and  partly using computer \cite{newae}, the following bound considerably improving \eqref{eq:2s-1}.
\begin{theorem}\label{thm:approach2}
   For all $s$, except $s=4,7,10$ there is
   $$\alphahat(s) \geq \lfloor \sqrt{2.5s} \rfloor.$$
\end{theorem}
   Lower bounds on the Waldschmidt constant combined with a simple condition count quickly lead to
   the following result generalizing classical Chudnovsky's Theorem for points in $\PP^2$.
\begin{theorem}[A Chudnovsky-type result for very general lines]\label{thm:chud}
   For all $s\geq 1$ there is
   $$\alphahat(s) \geq \frac{\alpha(s) + 1}{2}.$$
\end{theorem}

We also provide an algorithm $L$, which gives even better bounds for $\alphahat(s)$. This algorithm runs for each $s$ separately.
It seems not feasible
to write a closed formula for the output of the algorithm. However, we compare in Table \ref{tab:compare} the bounds resulting from various approaches.
\renewcommand*{\arraystretch}{1.15}
\begin{table}[h]\label{tab:compare}
$$\begin{array}{c|cccccccc}
   s                                         &   10   &   20   &   50   &   100   &   200   &    300    &  400     &     500  \\ \hline
\text{Theorem \ref{thm:chud}}                &   3.5  &   6    &    8   &    12   &    17   &   20.5    &   24     &      27  \\
\text{Theorem \ref{thm:approach1}}           &    4   &   6    &    9   &    14   &    19   &    24     &   28     &      31  \\   
\text{Theorem \ref{thm:approach1alg}}        &    4   &   6    &   10   &    14   &    20   &    24     &   28     &      31  \\
\text{Theorem \ref{thm:approach2alg}}        &    4   &   6    &   10   &    15   &    22   &    27     &   31     &      35  \\
\text{algorithm $L$}                         & 4.807  & 7.072  & 11.570 & 16.636  &   23.8  &  29.301   &  33.938  &   38.022 \\
\text{expected value }\;  e_s                & 5.107  & 7.388  & 11.899 & 16.977  & 24.154  &  29.660   &  34.302  &   38.392 \\
\end{array}$$
\caption{}
\end{table}

\section{The method}
   Our approach builds upon the upper semi-continuity of the dimension of cohomology groups.
   More precisely, in order to provide a lower bound on the Waldschmidt constant
   of a union of very general flats one needs to show that certain linear systems
   with prescribed vanishing order along the flats are empty, or actually stably empty,
   see Definition \ref{def:stably empty semi-eff}. As this is difficult to show
   for flats in a very general position directly, we specialize them, to a favorably position where one or other
   kind of induction process can be used. If the systems with flats in a special position are empty,
   then the same holds true for systems with flats in a very general position, this is exactly the
   yoga of the semi-continuity. See \cite{FMX18} for a very nice and precise discussion
   of this idea.

\section{Waldschmidt constants for lines -- the first approach}
   In this section we prove Theorems \ref{thm:approach1alg} and \ref{thm:approach1}.
\subsection{Proof of Theorem \ref{thm:approach1alg}}
   We assume to the contrary that there exists a divisor $D$ of degree $d$, with multiplicities at least $m$ along all $s$ lines such that
   $d/m < q$, cf. (\ref{ZN}). Then $d\leq qm-1$.
   We specialize $k^2$ out of $s$ very general lines onto $k$ general planes, $k$ lines on each of $k$ planes.

   Let $H$ be one of the fixed planes. If $H$ is not a component of $D$, then the restriction of $D$ to $H$ vanishes to order $m$
   along the $k$ lines in $H$. Subtracting these lines from $D\restr{H}$ we obtain
   a curve of degree $d-km\leq (q-k)m-1$ which passes through $s-k^2\geq (q-k)^2$ very general points
   with multiplicity $m$. Since the Nagata Conjecture holds for the square number of points (here $(q-k)^2$ points),
   this is a contradiction, cf. \cite[Remark 2.6]{DHST14}.

Hence, all distinguished planes are components of $D$. Subtracting them from $D$ we obtain a divisor of degree $d-k$
vanishing along each of specialized lines to order $m-1$. This divisor restricted to $H$ after removing its line components
has degree $d-km = (d-k) - k(m-1)$. Additionally it
has multiplicity $m$ at each of $s-k^2$ very general points in $H$. Hence $H$ must be again its component. Continuing in this
way we obtain a contradiction with the existence of $D$.

\subsection{Proof of Theorem \ref{thm:approach1}}
   Let $s\geq 1$ be fixed and let $q=\lfloor \sqrt{2s-1} \rfloor$. We claim that there exists an integer $k$ satisfying
$$(q-k)^2 \leq s-k^2.$$
   Indeed, the quadratic function
$$f(k)=2k^2-2qk+q^2-s$$
   attains its minimum at $k_0=q/2$. Since $q^2\leq 2s-1$, we have $f(q/2+1/2) \leq 0$. Thus
   $f$ is non-positive on an interval of length at least $1$ (from $(q-1)/2$ to $(q+1)/2$).
   Hence there exists in this interval an integer $k$ such that $f(k) \leq 0$. The assertion then follows from Theorem \ref{thm:approach1alg}.

\section{Waldschmidt constants for lines -- the second approach}
   We begin with a preparatory statement dealing with divisors in $\PP^2$.
\begin{lemma}\label{degeneration}
    Let $s$, $k$ and $q>k$ be nonnegative
     integers satisfying $(q-k)^2 \leq s-k$ and $qk\leq s$.
    Consider $q-1$ very general lines $L_{1},\dots,L_{q-1}$ in $\PP^2$, each containing $k$ distinguished very general points and $s-qk$ additional
    very general points on $\PP^2$, so that there are altogether $s-k$ distinguished points. Let $\Gamma$ be a divisor vanishing
    to order at least $m$ at all these points. Then
    $$\deg(\Gamma)\geq (q-k)m.$$
\end{lemma}

\begin{proof}
   Assume that there exists a divisor $\Gamma$ with $\deg(\Gamma)\leq (q-k)m-1$. The proof splits in two cases, depending on the applicability of Bezout's Theorem.

\textbf{Case $q\leq 2k$}.
   If $L_i$ is not a component of $\Gamma$, then
   $$km-1\geq (q-k)m-1\geq (\Gamma.L_i)\geq km,$$
   a contradiction. Hence all the lines $L_1,\dots,L_{q-1}$ are
   components of $\Gamma$ by Bezout's Theorem.
   The divisor $\Gamma-L_1-\ldots-L_{q-1}$ has degree at most $(q-k)m-1-(q-1)\leq (q-k)(m-1)-1$ and has the multiplicity at least $m-1$ in each of the points.
   Repeating the argument with $m$ replaced by $m-1$, we conclude that $\Gamma$ contains $m(L_1+\ldots+L_{q-1})$. This is a contradiction.

\textbf{Case $q>2k$}.
   We take additional very general lines $M_1,\dots,M_{q-k}$ in $\PP^2$. In particular they do not pass through any intersection point
   $L_i\cap L_j$ for $1\leq i<j\leq q-1$. Now, we specialize distinguished points on lines $L_1,\dots,L_{q-1}$, so that they become
   intersection points between the lines $L_i$ and $M_j$ and also the remaining points get specialized on lines $M_j$.
   It can be arranged so that there are altogether $q-k$ points on each $M_j$. This is possible, since we may specialize any point on $L_i$
   to arbitrary $M_j$ (it is important in this case that the number of points $k$ we want to specialize is smaller
   than the total number of intersections of $L_i$ with $M_1,\dots,M_{q-k}$, which is equal to $q-k$).
   In this construction we need altogether at least $(q-k)^2$ points and this number of points
   is guaranteed by the assumptions.

   Intersecting each of the lines $M_j$ with $\Gamma$, we see by Bezout's Theorem that now these lines must be components of $\Gamma$.
   Subtracting their union from $\Gamma$ results in a divisor with degree $q-k$ less than the degree of $\Gamma$ and multiplicities
   at all points at least $m-1$. It follows as before, that $\Gamma$ contains $m(M_1+\ldots+M_{q-k})$ which is not possible.
\end{proof}

\begin{theorem}\label{empty}
   Let $I$ be the ideal of $s$ very general lines in $\PP^3$. Let $m$ and $q$ be some fixed positive
   integers and assume that there is an integer $k$ such that
   $qk \leq s$ and $(q-k)^2 \leq s-k$. Then $\alpha(I^{(m)})\geq qm$.
\end{theorem}

\begin{proof}
   It suffices to show that there is no divisor $D$ of degree $\leq qm-1$ vanishing to order $m$
   along \emph{some} $s$ lines. Let $H_1,\ldots,H_q$
   be general planes in $\PP^3$. We specialize $k$ lines onto each of these planes. We assume,
   to the contrary that in this situation a divisor $D$ as above exists.

   Assume furthermore that $H_1$ is not a component of $D$. Then the trace of $D$ on $H_1$
   is a divisor vanishing with multiplicity $m$ along each line in $H_1$. Subtracting these
   lines from $D\restr{H_1}$ we get a divisor $\Gamma$ of degree $\leq (q-k)m-1$
   vanishing to order $m$ at intersection points of $H_1$ with the remaining $s-k$ lines.
   Note that for example the intersection points of lines in $H_2$ with $H_1$ are general
   points on the line $H_1\cap H_2$. Lemma \ref{degeneration} implies then that $\Gamma$
   does not exist. Hence $D$ contains each of the planes $H_1,\ldots,H_q$ as a component.
   Subtracting them from $D$ we obtain a divisor of degree $\leq q(m-1)-1$ vanishing to order at least
   $m-1$ along all lines. Thus the same argument can be repeated with $m$ replaced by $m-1$.
   Proceeding by induction we show that $D$ contains $qm$ planes, a contradiction.
\end{proof}
   As an immediate Corollary we obtain Theorem \ref{thm:approach2alg}.
\section{An algorithm to bound Waldschmidt constants for lines in $\PP^3$}\label{sec:An algorithm}
   Theorem \ref{empty} opens door to an algorithmic approach to bounding
   Waldschmidt constants for lines.
   We establish first the notation.
   We write $\call_N(d;m_1,\dots,m_s)$ to denote the linear system of
   divisors of degree $d$ in $\PP^N$ with multiplicities at least $m_j$ at
   given very general points if $N=2$ or very general lines if $N=3$.
   By a slight abuse of notation, we use the same symbol with rational coefficients to
   denote $\QQ$-divisors. This does no harm since we are interested in asymptotic properties
   of considered linear systems.

   We write $\call_3(d;\overline{m_1,\dots,m_r},m_{r+1},\dots,m_s)$ to denote
   the linear system $\call_3(d;m_1,\dots,m_s)$ with the $r$ first lines specialized
   to lines in one ruling of a fixed smooth quadric $Q \subset \PP^3$. The remaining lines
   are assumed to be in a very general position. We write $m^{\times u}$ to abbreviate
   $u$ occurrences of $m$ in the tuple, for example $\call_N(6;1,1,1,2,2,3)=\call_N(6;1^{\times 3},2^{\times 2},3)$.
\begin{definition}[Stably empty and semi-effective]\label{def:stably empty semi-eff}
   We say that the system $\call_N(\delta;q_1,\dots,q_s)$, with $\delta,q_1,\dots,q_s \in \QQ$, is \emph{stably empty}
   if the linear systems $\call_N(d;mq_1,\dots,mq_s)$ are empty for all $d \leq \delta m$ and all $m$ such that $d,mq_1,\dots,mq_s$ are integers.
   We say that $\call_N(\delta;q_1,\dots,q_s)$ is \emph{semi-effective} if it is not stably empty. Finally, we say that
   $\call_N(\delta;q_1,\dots,q_s)$ is \emph{integral} if all numbers involved in the sequence are integers.
\end{definition}
\begin{remark}\label{goodm}
   The notion of semi-effective (also known as $\QQ$-effective) divisors has been introduced by Harbourne \cite[Definition 2.2.1]{recent}.
   A $\QQ$-divisor $D$ is semi-effective if there is an $m$ such that $mD$ is integral and effective.
   Both definitions are equivalent. Indeed, by assumption there exist $d$ and $k$ such that $d\leq \delta k$ and $\call_N(d;kq_1,\dots,kq_s)$ is integral and non-empty.
   Let $h$ be the denominator of $\delta$, obviously the system $\call_N(dh;khq_1,\dots,khq_s)$ is non-empty. Since $dh\leq kh\delta$,
   the system $\call_N(kh\delta;khq_1,\dots,khq_s)$ is integral and non-empty as well. So the claim holds with $m=kh$.
\end{remark}
   We have the following easy observation.
\begin{lemma}\label{Wald}
   For any rational number $\delta>\alphahat(s)$ the system $\call_3(\delta;1^{\times s})$ is semi-effective.
\end{lemma}

\begin{proof}
   By the definition of the Waldschmidt constant, there exist $d$ and $m$ such that $d/m < \delta$ and the linear system $\call_3(d;m^{\times s})$ is non-empty. Therefore the claim follows.
\end{proof}

\begin{lemma}
\label{3op}
   Let $\call_2(\delta;q_1,\dots,q_s)$ be semi-effective. Then
\begin{enumerate}
\item
$\call_2(\delta;q_{\sigma(1)},\dots,q_{\sigma(s)})$ is semi-effective for any permutation $\sigma\in\Sigma_s$;
\item
For $k=\delta-q_1-q_2-q_3$, $\call_2(\delta+k;q_1+k,q_2+k,q_3+k,q_4,\dots,q_s)$ is semi-effective;
\item
If $q_1=q_2=q_3=q_4$ then $\call_2(\delta;2q_1,q_5,q_6,\dots,q_s)$ is semi-effective.
\end{enumerate}
\end{lemma}

\begin{proof}
The first claim is obvious. By Remark \ref{goodm} there exists $m$ such that $\call_2(m\delta;mq_1,\dots,mq_s)$ is integral and non-empty.
A standard Cremona transformation of $\PP^2$, applied to this system, gives the non-empty system
$\call_2(m(\delta+k);m(q_1+k),m(q_2+k),m(q_3+k),mq_4,\dots,mq_s)$, hence the second claim follows.
Since $\call_2(2mq_1-1;mq_1^{\times 4})$ is empty, if $\call_2(m\delta;2mq_1,mq_5,\dots,mq_s)$ were empty,
then by \cite[Theorem 1]{Dum09} the system $\call_2(m\delta;mq_1,mq_1,mq_1,mq_1,mq_5,\dots,mq_s)$ would be empty. This gives the third claim.
\end{proof}
   We describe now the algorithm $T$. Its input is $(\delta;q_1,\dots,q_s;p)$: an $(s+1)$-tuple of rational numbers extended by an integer $p$.
   Let $q = \sum_{j=1}^{s} q_j$. With the input data we associate the system
\begin{equation}
\label{syst}
   \call_2(2\delta-q+(s-4)t;\delta-2t,\delta-q+(s-2)t,1^{\times 2p}),
\end{equation}
where $t$ is an indeterminate; we begin with this system, and, during the procedure, we will alter the entries, which are
elements in $\QQ[t]$. Now fix some small $\tau \in \QQ$, $\tau > 0$. The power of the algorithm strongly depends on choosing $\tau$. Smaller $\tau$ gives better results, but forces the algorithm to take more time.

We will use $\tau$ to order elements in $\QQ[t]$. Namely, we define that $f>g$ if $f(\tau)>g(\tau)$.
Then we perform the following procedure. In all steps we deal with a system of the form $\call_2(d(t);m_1(t),\dots,m_r(t))$.
The first term $d(t)$ will be called the degree, the others will be called multiplicities. If $m(\tau)\leq 0$ during computations,
then it is omitted in the next step.
\begin{procedure}[Algorithm $T$]\rm $ $
\begin{itemize}[leftmargin=0.5cm]
   \item Step 1. Sort multiplicities in the non-increasing order, using the ordering given above. If $m_j(\tau)\leq0$ then put $m_j = 0$.
   \item Step 2. If there are at least three non-zero multiplicities, compute $k(t)$ equal to the degree minus the sum of the three greatest multiplicities. If $k(\tau)<0$, then add $k(t)$ to the degree and to the three greatest multiplicities, as in point 2) of Lemma \ref{3op}; then go to Step 1.
   \item Step 3. Find four equal multiplicities in the sequence and replace them by twice the value of this multiplicity, as in point 3) of Lemma \ref{3op}; then go to Step 1.
\end{itemize}
   If neither Step 2 nor Step 3 can be performed, then the algorithm terminates. Observe that in each Step the degree and multiplicities are linear combinations, with integer coefficients, of the input data. Thus there exists a constant $\beta > 0$ such that if $k(\tau)<0$ then $k(\tau) \leq -\beta$. Consequently Step 2 and Step 1 cannot be performed infinitely many times, since each time (in Step 2) the evaluation at $\tau$ of three multiplicities decreases by at least $\beta$, and in Step
   1 a multiplicity is set to zero if its evaluation at $\tau$ becomes negative.

\end{procedure}
   Assume that the degree after the termination of the procedure is equal to $a+bt$ (only affine operations to the degree were performed). Then the algorithm $T$ returns
$$t_0=T(\delta;q_1,\dots,q_s;p)=\left\{ \begin{array}{lcl}
0 & \mbox{ if } & a \geq 0, \\
\min\{q_1,\dots,q_s\} & \mbox{ if } &  a < 0 \text{ and } b \leq 0, \\
\min\{-a/b,q_1,\dots,q_s\} & &\mbox{otherwise. }  
\end{array}\right.$$

The following example illustrates Algorithm $T$ for input data $(7;1,1,1,1,1;15)$.

\begin{example}
Let $\tau=1/1000$.
The associated system is $\call_2(9+t;7-2t,2+3t,1^{\times 30})$. In each line we write the system after performing Step 1
(sort and kill negative multiplicities). We also write $k(t)$ for each system to recognize if
Step 2 (for $k(\tau)<0$) or Step 3 (otherwise) is performed.
$$\begin{array}{rll}
\call_2(9+t;&7-2t,2+3t,1^{\times 30}) &  k(t)=-1\\
\call_2(8+t;&6-2t,1+3t,1^{\times 29}) &  k(t)=0\\
\call_2(8+t;&6-2t,2,1+3t,1^{\times 25}) &  k(t)=-1\\
\call_2(7+t;&5-2t,1^{\times 26},3t)  & k(t)=3t\\
\call_2(7+t;&5-2t,2,1^{\times 22},3t) &  k(t)=-1+3t\\
\call_2(6+4t;&4+t,1+3t,1^{\times 21},3t^{\times 2})  & k(t)=0\\
\call_2(6+4t;&4+t,2,1+3t,1^{\times 17},3t^{\times 2}) &  k(t)=-1\\
\call_2(5+4t;&3+t,1^{\times 18},3t^{\times 3}) &  k(t)=3t\\
\call_2(5+4t;&3+t,2,1^{\times 14},3t^{\times 3}) &  k(t)=-1+3t\\
\call_2(4+7t;&2+4t,1+3t,1^{\times 13},3t^{\times 4}) &  k(t)=0\\
\call_2(4+7t;&2+4t,2,1+3t,1^{\times 9},3t^{\times 4}) &  k(t)=-1\\
\call_2(3+7t;&1+4t,1^{\times 10},3t^{\times 5}) &  k(t)=3t\\
\call_2(3+7t;&2,1+4t,1^{\times 6},3t^{\times 5}) &  k(t)=-1+3t\\
\call_2(2+10t;&1+3t,1^{\times 5},7t,3t{\times 6})  & k(t)=-1+7t\\
\call_2(1+17t;&1^{\times 3},10t,7t^{\times 3}, 3t^{\times 6})  & k(t)=-2+17t\\
\call_2(-1+34t;&10t,7t^{\times 3},3t^{\times 6})   & k(t)=-1+10t\\
\call_2(-2+44t;&7t,3t^{\times 6})  &  k(t)=-2+31t\\
\call_2(-4+75t;&3t^{\times 4})  &  k(t)=-4+66t\\
\call_2(-8+141t;&3t) &
\end{array}$$
The output is $\frac{8}{141}$.

\end{example}

\begin{lemma}\label{tT}
   Let $(\delta;\ell_1,\dots,\ell_s;p)$ be as above and let $t_0$ be the output of algorithm $T$.
   Then \eqref{syst} is stably empty for all rational $t$ in the range $0 \leq t < t_0$.
\end{lemma}

\begin{proof}
   Assume that \eqref{syst} is semi-effective for some $0\leq t <t_0$.
   By Lemma \ref{3op}, the final sequence in $T$, with the first entry equal to $a+bt$, is semi-effective.
   Then it must be
   \begin{equation}\label{eq:ab}
      a+bt\geq 0,
   \end{equation}
   since the degree of a non-empty system, equal to $(a+bt)m$, must be nonnegative.
   For $t_0 = 0$ there is nothing to prove, so let $a<0$.
   If $b\leq 0$ then $a+bt \leq a < 0$, a contradiction with \eqref{eq:ab}. If $b > 0$ then $a+bt < a+bt_0 \leq 0$, again a contradiction with \eqref{eq:ab}.
\end{proof}

\begin{lemma}\label{trace}
   For $V=\call_3(d;\overline{m_1,\dots,m_s},m^{\times p})$ let
   $\mu = \sum_{j=1}^{s} m_j$.
   Assume that a quadric $Q$ is not a fixed component of $V$. Then the trace of $V$ on $Q$
   can be viewed under the standard birational map from $Q$ to $\PP^2$ as the linear system
   $$W=(2d-\mu;d,d-\mu,m^{\times 2p})$$
   on $\PP^2$. If $V$ is non-empty, so is $W$.
\end{lemma}

\begin{proof}
   The proof is classical and can be found in \cite[Proposition 15]{Dum10}, see also \cite{DVL07}. We present a sketch for reader's convenience.
   The quadric $Q$ is isomorphic to $\PP^1 \times \PP^1$.
   The restriction of a divisor of degree $d$ in $\PP^3$ to $Q$ (if $Q$ is not a component of this divisor) is a divisor $\Gamma$
   on $\PP^1 \times \PP^1$ of bidegree $(d,d)$. The $s$ lines with multiplicities $m_1,\dots,m_s$ are components of $\Gamma$.
   Subtracting them from $\Gamma$ we obtain a
   divisor $\Gamma'$ of bidegree $(d-\mu,d)$. The remaining $p$ very general lines intersect $Q$ in $2p$ points.
   The divisor $\Gamma'$ must vanish at these points to order at least $m$.
   It maps to $\PP^2$ to an effective divisor of degree $2d-\mu$, with two additional points
   of multiplicity $d-\mu$ and $d$, and $2p$ points with multiplicity $m$.
\end{proof}

\begin{lemma}\label{t0out}
   Let $\call_3(\delta;\overline{q_1,\dots,q_s},1^{\times p})$ be semi-effective.
   Let $t_0 = T(\delta;q_1,\dots,q_s;p)$. Then
$$\call_3(\delta-2t_0;\overline{q_1-t_0,\dots,q_s-t_0},1^{\times p})$$
is semi-effective.
\end{lemma}

\begin{proof}
Let $\call_3(m\delta;\overline{mq_1,\dots,mq_s},m^{\times p})$ be integral and non-empty. Without loss of generality
we may assume that $mt_0$ is integral.

   Assume that $Q$ is contained as a $k_0$-fold base component of this system, so that the residual system
   $$\call_3(m\delta-2k_0;\overline{mq_1-k_0,\dots,mq_s-k_0},m^{\times p})$$
is non-empty and $Q$ is not its base component.
We want to prove that $k_0 \geq mt_0$.

If $k_0$ is greater than or equal to the minimum of $mq_1,\dots,mq_s$, then we are done, since $t_0 \leq \min\{q_1,\dots,q_s\}$.
In the opposite case the multiplicities $mq_j-k_0$ are nonnegative.

Let $q = \sum_{j=1}^{s} q_j$. By Lemma \ref{trace} the residual system restricted to $Q$
and transferred to $\PP^2$ gives a non-empty system
$$\call_2(2m\delta-mq+(s-4)k_0;m\delta-2k_0,m\delta-mq+(s-2)k_0,m^{\times 2p}).$$
Dividing by $m$, for $t=k_0/m$ we obtain a semi-effective system on $\PP^2$
$$\call_2(2\delta-q+(s-4)t;\delta-2t,\delta-q+(s-2)t,1^{\times 2p}).$$
Since $t_0$ is the outcome of $T$, Lemma \ref{tT} implies that $t \geq t_0$. Thus $k_0 \geq mt_0$.

   It follows that the system $\call_3(m\delta;\overline{mq_1,\dots,mq_s},m^{\times p})$
   contains $Q$ as a base component with multiplicity at least $mt_0$.
   Subtracting this base component, we get
   the non-empty system $\call_3(m(\delta-2t_0);\overline{m(q_1-t_0),\dots,m(q_s-t_0)},m^{\times p})$.
   This proves the assertion.
\end{proof}

   We now define our second algorithm, Algorithm $L$.
   It works with sequences $(\delta;\overline{q_1,\dots,q_s},1^{\times p})$ of rational numbers $\delta,q_1,\dots,q_s$ and an integer $p$.
   During the procedure, these numbers will be altered. As before, we fix a small $\tau > 0$.
\begin{procedure}[Algorithm $L$]\label{proc:L}\rm $ $
\begin{itemize}[leftmargin=0.5cm]
   \item Step 1. Check if $\delta < 1$ and $p\geq 1$; or $\delta < q_j$ for some $j$. If so, return ``yes'' and finish.
   \item Step 2. Run Algorithm $T$ to get $t_0 = T(\delta;q_1,\dots,q_s;p)$. If $t_0\geq \tau$ define the new sequence
      $(\delta-2t_0;\overline{q_1-t_0,\dots,q_s-t_0},1^{\times p})$ and go to Step 1.
   \item Step 3. If $p>0$ then define the new sequence $(\delta;\overline{q_1,\dots,q_s,1},1^{\times (p-1)})$ and go to Step 1.
   \item Step 4. Answer ``no''.
\end{itemize}
\end{procedure}
Observe that the algorithm must terminate, since in Step 2 the number $\delta$ decreases by at least $2\tau$ (and $\delta < 0$ certainly
finishes Algorithm $L$), and in Step 3 the number $p$ decreases by 1 ($p=0$ also finishes the algorithm).

The following example illustrates Algorithm L for input data $(4;1^{\times 8})$.
\begin{example}
Let $\tau=1/1000$. In each line we write a system at the beginning of Step 1 and $t_0$ given by Algorithm T.
$$\begin{array}{rll}
(4;&1^{\times 8})&\\
(4;&\overline{1},1^{\times 7})&  t_0=0\\
(4;&\overline{1,1},1^{\times 6})&  t_0=0\\
(4;&\overline{1,1,1},1^{\times 5})&  t_0=4/7\\
(20/7;&\overline{3/7,3/7,3/7},1^{\times 5})&   t_0=3/14\\
(17/7;&\overline{3/14,3/14,3/14},1^{\times 5})&  t_0=27/224\\
(35/16;&\overline{3/32,3/32,3/32},1^{\times 5})&  t_0=135/2464\\
(160/77;&\overline{3/77,3/77,3/77},1^{\times 5})&   t_0=115/4928\\
(65/32;&\overline{1/64,1/64,1/64},1^{\times 5})&   t_0=49/5184\\
(163/81;&\overline{1/162,1/162,1/162},1^{\times 5})&   t_0=751/200394\\
(2480/1237;&\overline{3/1237,3/1237,3/1237},1^{\times 5})&   t_0=11424/6240665\\
(10096/5045;&\overline{3/5045,3/5045,3/5045},1^{\times 5})&   t_0=0\\
(10096/5045;&\overline{3/5045,3/5045,3/5045,1},1^{\times 4})&   t_0=3/5045\\
(2;&\overline{5042/5045},1^{\times 4})&   t_0=5042/5045\\
(6/5045;&1^{\times 4})&
\end{array}$$
Answer "yes".
\end{example}

\begin{lemma}
If Algorithm $L$ performed on $(\delta;1^{\times s})$ returns ''yes'', then
   $$\alphahat(s)\geq \delta.$$
\end{lemma}

\begin{proof}
   Assume to the contrary that $\alphahat(s) < \delta$. By Lemma \ref{Wald}, $(\delta;1^{\times s})$ is semi-effective.
   We run algorithm $L$ on this sequence.
   The Steps 2 and 3 transform semi-effective
   sequences into semi-effective sequences. For Step 2 we use Lemma \ref{t0out}, for step 3 observe that if a system with a line
   in a very general position is non-empty, then it is also non-empty for this line specialized to $Q$.

   By our assumption, the Algorithm $L$ finishes with ''yes''. This means that the system
   $\call_3(\tilde\delta;\overline{q_1,\dots,q_s},1^{\times p})$ with $\tilde\delta < 1$ and $p \geq 1$, or $\tilde\delta < q_j$ is semi-effective.
   This is a contradiction, since a nonempty system cannot have a degree strictly lower than the order of its vanishing along a line.
\end{proof}
   We use now our considerations in this section to prove Theorem \ref{thm:approach2}.
\subsection{Proof of Theorem \ref{thm:approach2}}
   For $s$ large enough, Theorem \ref{thm:approach2} follows from Theorem \ref{thm:approach2alg}.
   Indeed, let
   $q:= \lfloor\sqrt{2.5s}\rfloor$ and $k:=\lfloor\sqrt{0.4s}\rfloor$.
   Then $qk\leq s$ holds obviously. For the second condition in Theorem \ref{thm:approach2alg}
   we use the stronger inequality
   $$(\sqrt{2.5s}-\sqrt{0.4s}+1)^2\leq s-\sqrt{0.4s},$$
   which holds for $s\geq 490$. This can be checked elementarily.

   For lower values of $s$ we use computer to run Procedure \ref{proc:L} with $\delta=\lfloor\sqrt{2.5s}\rfloor$.
   It verifies the assertion for all values of $s$ except $4,7$ and $10$. Since for $s=4$ we have
   $\alphahat(4)=8/3<\sqrt{2.5\cdot 4}$, the assertion cannot hold. For $s=7$ the situation is more
   complicated. We have $e_7\simeq 4.203503$ and $\sqrt{2.5\cdot 7}\simeq 4.1833$, so that the
   assertion might hold. In fact, it is expected that $\alphahat(7)=4.2$. Our algorithm returns only $3.837$ as the
   lower bound in this case. For $s=10$ we have
   $e_{10}\simeq 5.107249$, whereas $\sqrt{2.5\cdot 10}=5$, so the assertion might hold, but its
   proof would require some more refined methods, since our algorithm returns only $4.807$ in this case.

\section{A Chudnovsky-type result}
   In this section we derive Theorem \ref{thm:chud} from lower bounds on $\alphahat(s)$.
\begin{lemma}
\label{assq}
Let $a$, $s$ be integers satisfying $a\geq 10$ and $(a+2)(a+1)\leq 6s$. Then
\begin{equation}\label{eqc1}
\sqrt{2s-1}-1 \geq \frac{a+1}{2}.
\end{equation}
\end{lemma}

\begin{proof}
After elementary operations we get
the equivalent inequality
$$8s \geq a^2+6a+13.$$
Since, by assumption, $8s \geq 4/3(a+1)(a+2)$, it is enough to show that
$$\frac{4}{3}(a+1)(a+2) \geq a^2+6a+13,$$
which holds for $a\geq 10$.
\end{proof}

Finally we prove Theorem \ref{thm:chud}.
\subsection{Proof of Theorem \ref{thm:chud}}
   Since there exists no divisor of degree $\alpha(s)-1$ vanishing along $s$ very general lines, counting conditions we see that
   it must be (cf. \cite[Lemma 2.1]{DHST14}).
\begin{equation}
\label{eqal}
\binom{(\alpha(s)-1)+3}{3} \leq s((\alpha(s)-1)+1).
\end{equation}
   This is equivalent to $(\alpha(s)+2)(\alpha(s)+1) \leq 6s$. Now, by Theorem \ref{thm:approach1} and Lemma \ref{assq}
   $$\alphahat(s) \geq \lfloor \sqrt{2s-1} \rfloor \geq \sqrt{2s-1}-1 \geq \frac{\alpha(s)+1}{2}$$
   for $\alpha(s) \geq 10$. Hence for $s\geq 22$ we are done.
   For $s=1,3,4,\dots,21$ we compare $\lfloor \sqrt{2s-1} \rfloor$ with $\frac{\alpha(s)+1}{2}$ for $\alpha(s)$ satisfying \eqref{eqal}
to get the result. For $s=2$ we get the bound $\alphahat(s) \geq 2$ by Theorem \ref{thm:approach1alg} ($k=1$) although
$\lfloor \sqrt{2\cdot 2-1} \rfloor = 1$.

\section{The limits of the method}
   As already mentioned the upper bound $\alphahat(s)\leq e_s$ has been proved in \cite[Theorem 2.5]{DHST14}
   and it is conjectured in \cite[Conjecture A]{DHST14} that equality holds for $s$ sufficiently large.
   In the present note, we specialize the lines so that there arise intersection points between them.
   It has been discussed in \cite[Example 20]{ahp1} and generalized in \cite[Example 5]{ahp2} that in case
   of intersecting lines there is a correction term in the coefficients of $\Lambda_s$. We have
   $\Lambda_2(t)=t^3-6t+4$ for a pair of skew lines, whereas for a pair of intersecting lines
   we have $\Lambda_{2,1}(t)=t^3-6t+6$. We omit a technical and not interesting here proof of the
   fact that for $s$ lines with $k$ simple intersection points (at most two lines meet in a point) we have
   $$\Lambda_{s,k}(t)= t^3-3st+2s+2k.$$
   It is expected, see \cite[Conjecture 13]{ahp2}, that also in this case, for $s$ sufficiently large,
   the Waldschmidt constant of the arrangement of $s$ lines with $k$ simple intersection points is
   equal to the largest real root of the asymptotic Hilbert polynomial $\Lambda_{s,k}(t)$. As this
   root is slightly smaller than the root of $\Lambda_s(t)$, our method can never prove that
   $\alphahat(s)=e_s$. However the bound we get is very close and thus of interest.
\begin{example}
   Let $s=100$. By Theorem \ref{thm:approach2alg} with $k=6$ we get $\alphahat(100)\geq 15$. The specialization we have made
   (putting lines onto $15$ planes, $6$ lines on each plane) generates $225$ simple intersection points.
   The Waldschmidt constant for this configuration cannot exceed $16.114$, the largest root of $\Lambda_{100,225}(t)=t^3-300t+650$,
   whereas the largest root of $\Lambda_{100}(t)=t^3-300t+200$ is $16.977$.
\end{example}

\paragraph*{Acknowledgement.}
   This research has been carried out while the second author was visiting
   as a senior graduate student the Department of Mathematics of the Pedagogical University of Cracow.
   The first, the third and the fourth authors were partially supported by
   National Science Centre, Poland, grant 2014/15/B/ST1/02197.
   We thank Tomasz Szemberg for helpful remarks.

\bibliographystyle{abbrv}
\bibliography{master}

\bigskip \small

\bigskip
Marcin Dumnicki, Halszka Tutaj-Gasi\'nska,
Jagiellonian University, Faculty of Mathematics and Computer Science, {\L}ojasiewicza 6, PL-30-348 Krak\'ow, Poland

\nopagebreak
\textit{E-mail address:} \texttt{Marcin.Dumnicki@uj.edu.pl}
\textit{E-mail address:} \texttt{Halszka.Tutaj-Gasinska@uj.edu.pl}

\bigskip
   Mohammad Zaman Fashami,
   Faculty of Mathematics, K. N. Toosi University of Technology, Tehran, Iran.

\nopagebreak
   \textit{E-mail address:} \texttt{zamanfashami65@yahoo.com}

\bigskip
   Justyna Szpond,
   Department of Mathematics, Pedagogical University of Cracow,
   Podchor\c a\.zych 2,
   PL-30-084 Krak\'ow, Poland

\nopagebreak
   \textit{E-mail address:} \texttt{szpond@gmail.com}

\end{document}